\documentclass[smallextended]{svjour3}      
\smartqed  
\usepackage{graphicx}
\usepackage{amsmath}
\usepackage{amsfonts}
\usepackage{amssymb}
\usepackage{xspace}
\usepackage{threeparttable}
\usepackage{subfigure}
\usepackage{bm}
\usepackage[ruled,lined,linesnumbered]{algorithm2e}

\usepackage{hyperref}
\hypersetup{
 colorlinks=true,
 citecolor=blue,
 linkcolor=blue,
 urlcolor=blue}

\DeclareMathAlphabet{\mathcal}{OMS}{cmsy}{m}{n}

\newcommand{\ignore}[1]{}

\newcommand{\mat}[1]{\ensuremath{\mathsf{#1}}}

\newcommand{\eq}[0]{h}
\newcommand{\ineq}[0]{g}

\newcommand{\me}[0]{m_{h}}
\newcommand{\mi}[0]{m_{g}}
\newcommand{\lag}[0]{\bm{\lambda}}
\newcommand{\lage}[0]{\lambda}
\newcommand{\lagi}[0]{\mu}

\newcommand{\Tr}[0]{^{\mathrm{T}}}

\newcommand{\eg}[0]{{e.g.\@}\xspace}
\newcommand{\ie}[0]{{i.e.\@}\xspace}

\DeclareMathOperator*{\argmin}{argmin}

\SetKwComment{mycomment}{}{}

\newtheorem{thrm}{Theorem}
\newtheorem{prop}{Proposition}

\renewcommand{\theenumi}{\Roman{enumi}}

\smartqed

\journalname{Optimization and Engineering}
\begin{document}

\title{Error-tolerant Multisecant Method for Nonlinearly Constrained Optimization\thanks{This work was supported by the Air Force Office of
    Scientific Research Award FA9550-15-1-0242 under Dr. Jean-Luc Cambier.}  }

\titlerunning{Multisecant Accelerated Descent}        

\author{Jason E Hicken \and Pengfei Meng \and Alp Dener}


\institute{ 
  Department of Mechanical, Aerospace, and Nuclear Engineering, \\
  Rensselaer Polytechnic Institute\\
  Troy, New York, 12180, United States\\
  Tel.: 1-518-276-4893\\
  \email{hickej2@rpi.edu}\\
  www.optimaldesignlab.com
}

\date{Received: date / Accepted: date}

\maketitle

\begin{abstract}

We present a derivative-based algorithm for nonlinearly constrained optimization problems that is tolerant of inaccuracies in the data.   The algorithm solves a semi-smooth set of nonlinear equations that are equivalent to the first-order optimality conditions, and it is matrix-free in the sense that it does not require the explicit Lagrangian Hessian or Jacobian of the constraints.  The solution method is quasi-Newton, but rather than approximating only the Hessian or constraint Jacobian, the Jacobian of the entire nonlinear set of equations is approximated using a multisecant method.  We show how preconditioning can be incorporated into the multisecant update in order to improve the performance of the method.  For nonconvex problems, we propose a simple modification of the secant conditions to regularize the Hessian.  Numerical experiments suggest that the algorithm is a promising alternative to conventional gradient-based algorithms, particularly when errors are present in the data.

\end{abstract}

\section{Introduction}\label{intro}

Gradient-based optimization can be a valuable tool in the engineering design process.  This is especially true when the target design is governed by complex (nonlinear) physics and there are many design parameters.  In such circumstances, even a seasoned engineer will be challenged to choose the parameters using intuition alone.  Exemplar applications of gradient-based optimization include aerodynamic shape optimization~\cite{Jameson1988aerodynamic,Reuther1999part1,Reuther1999part2,Nielsen1999aerodynamic,Anderson1999airfoil,Nemec2002newton,Nemec2004multipoint}, aerostructural design~\cite{Maute2001coupled,Martins2002coupled,Maute2003sensitivity,Kenway2013scalable,Kenway2014multipoint}, structural topology optimization~\cite{Bendsoe1988generating,Zhou1991coc,Sethian2000structural,Dijk2013levelset}, and satellite design \cite{Hwang2014large}, to name a few.

A practical challenge in applying gradient-based optimization is the efficient and accurate evaluation of derivatives.  Conventional optimization methods rely on accurate derivative information for computing search directions and for globalization strategies, \eg line-search and trust-region methods.  Inaccurate gradients and/or objectives produce inconsistencies during the globalization that ultimately lead to failure of the optimization algorithm.  The failure can occur well before the optimization has made significant progress toward improving the design or satisfying the constraints.

Given the importance of accurate and efficient gradients, it is hardly surprising that considerable work has been devoted to methods for differentiating engineering analysis codes.  For example, analytic sensitivity methods like the direct and adjoint senstivity methods~\cite{Pironneau1974optimum,Haftka1986structural,Jameson1988aerodynamic} are now used routinely in partial-differential equation (PDE) solvers, including some commerical codes~\footnote{For example, ANSYS CFD now provides adjoint capabilities}.  In addition, algorithmic differentiation~\cite{Griewank2008evaluating} has matured to the point that it can be applied to a large, complex PDE solver library~\cite{Albring2016efficient}.

Nevertheless, despite the progress in computing derivatives, there remain industrially relevant applications for which inaccuracies in the gradient are theoretically or practically unavoidable:
\begin{itemize}
  \item Problems governed by chaotic dynamics, in which the objective is a time-averaged quantity, cannot be treated using conventional sensitivity analysis methods~\cite{Lea2000sensitivity}.  Instead, methods like the least-square-shadowing (LSS) adjoint are needed~\cite{Wang2014least}; however, even the LSS adjoint produces gradients with errors that cannot be eliminated.
  \item Difficult nonlinear analyses often suffer from incomplete convergence of their algebraic solvers.  An example is the solution of the Reynolds-averaged Navier-Stokes equations for configurations at high lift conditions, where two to three orders reduction in the nonlinear residual may be acceptable.  Incomplete convergence leads to errors in both the objective function and the gradient.
  \item When the continuous-adjoint method~\cite{Pironneau1974optimum,Jameson1988aerodynamic} is used, the resulting gradient is inconsistent with the objective function in that it differs from the true gradient by errors on the order of the discretization~\cite{Collis2002analysis}.
  \item If the analysis mesh is regenerated during a line search, the gradient based on the previous mesh and the objective based on the new mesh will be inconsistent; this can occur when using a Cartesian adaptive flow solver~\cite{Aftosmis1997lecture,Aftosmis1998robust} for aerodynamic shape optimization~\cite{Nemec2006aerodynamic}.
  \item In a multifidelity analysis~\cite{Balabanov2004multi}, the objective may be evaluated using a high-fidelity model and the gradient may be evaluated using a low-fidelity model, again leading to inconsistencies.
\end{itemize}
While some of the above problems can be ameliorated in theory, \eg mesh refinement in the case of the continuous adjoint or mesh regeneration, the computational cost may not be acceptable in practice.

In this work, we present a multisecant quasi-Newton algorithm designed to address derivative-based optimization in the presence of inaccurate data.  Multisecant methods~\cite{Vanderbilt1984total,Eyert1996comparative} are a generalization of Broyden's method for nonlinear equations~\cite{Broyden1965class}, and they have been shown to be particularly effective for solving nonlinear equations that are high dimensional, expensive to evaluate, and potentially noisy~\cite{Bierlaire2006solving,Chelikowsky1996quantum,Fang2009two}.

To solve constrained optimization problems using multisecant methods, we formulate the first-order necessary optimality conditions as an equivalent set of nonlinear equations and apply the multisecant update directly to these equations.  An advantage of this approach is that it is matrix-free in the sense that it does not require the constraint Jacobian or Hessian of the Lagrangian.  This should be contrasted with conventional algorithms that require, at the least, the explicit constraint Jacobian.  The constraint Jacobian is especially problematic in PDE-constrained optimization problems with many state-dependent constraints, because each such constraint requires the solution of an adjoint equation.


The rest of the paper is organized as follows.  Section~\ref{sec:prelim} formally defines the optimization problem and shows how the first-order necessary conditions can be recast as an equivalent set of semi-smooth nonlinear equations.  Section~\ref{sec:algorithm} begins with a general review of multisecant methods, and then describes the particulars of our implementation, including how we incorporate preconditioning and how we handle nonconvexity.  Some numerical experiments are provided in Section~\ref{sec:results} to verify the method and demonstrate its effectiveness.  Finally, Section~\ref{sec:conclude} summarizes our findings and discusses future work.

\section{Preliminaries}\label{sec:prelim}

\ignore{
\subsection{Notation}

Vectors are represented with bold type.  We will reserve $\bm{x} \in
\mathbb{R}^{n}$ for the primal variables, which are sometimes called the design or control variables in the engineering optimization literature.  Lagrange multipliers, or dual variables, associated with the constraints will be denoted by $\lag \in \mathbb{R}^{m}$.

Uppercase letters in san-serif font are used to represent matrices.  For
example, $\mat{A} \in \mathbb{R}^{n\times m}$ is an $n\times m$ matrix.  The
$n\times n$ identity matrix is denoted by $\mat{I}_{n}$ and the $n\times m$
matrix of zeros is denoted by $\mat{0}_{n\times m}$; the subscripts will be
dropped if the dimensions of the matrix can be inferred from the context.
}

\subsection{Problem Definition}

We consider nonlinear optimization problems of the form
\begin{equation}\label{eq:opt}\tag{\text{P}}
  \begin{alignedat}{2}
    &\min_{x} & \quad f(x)&, \\
    &\mathrm{s.t.} & \eq(x) &= 0, \\
    &              & \ineq(x) &\geq 0,
  \end{alignedat}
\end{equation}
where $x \in \mathbb{R}^{n}$ denotes the optimization variables, $f:\mathbb{R}^{n} \rightarrow \mathbb{R}$ is the objective, $\eq : \mathbb{R}^{n} \rightarrow \mathbb{R}^{\me}$ are the equality constraints, and $\ineq : \mathbb{R}^{n} \rightarrow \mathbb{R}^{\mi}$ are the inequality constraints.  We will assume that the objective and constraints are continuously differentiable.

Our strategy for solving \eqref{eq:opt} is based on Newton's method, which requires that we recast the problem as a system of nonlinear equations.  To that end, we recall that a solution to \eqref{eq:opt} must satisfy the first-order optimality conditions, also known as the Karush-Kuhn-Tucker (KKT) conditions.

\begin{thrm}[Karush-Kuhn-Tucker (KKT) conditions]\label{thrm:kkt}
  Let $x^*$ denote a local solution of \eqref{eq:opt}.  If the set of active
  constraint gradients, namely
  \begin{equation*}
  \left\{ \nabla \eq_i(x^*) \; | \: i=1,\ldots,\me \right\} \bigcup \left\{ \nabla \ineq_i(x^*) \; | \; \ineq_i(x^*) =0,\; i = 1,\ldots,\mi \right\},
  \end{equation*}
  is linearly independent, then there are multiplier values $\lage^{*}$ and $\lagi^{*}$ such that 
  \begin{subequations}\label{eq:kkt}
    \begin{align}
     \nabla L(x^*,\lage^*,\lagi^*) & = 0, \label{eq:kkt_opt} \\
      \eq(x^*) &= 0, \label{eq:kkt_eq_feas} \\
      \ineq_i(x^*) \lagi_i^* &= 0, \quad \forall i = 1,\ldots,\mi, \label{eq:kkt_comp} \\
      \ineq_i(x^*) &\geq 0, \quad \forall i = 1,\ldots,\mi, \label{eq:kkt_ineq_feas} \\
      \lagi_i^{*} &\geq 0, \quad \forall i = 1,\ldots,\mi, \label{eq:kkt_mult_bound}
    \end{align}
  \end{subequations}
  where $L(x,\lage,\lagi) \equiv f(x) - \eq(x)^T \lage - \ineq(x)^T\lagi$ is the Lagrangian.
\end{thrm}

Newton's method cannot be applied directly to the KKT conditions, because it cannot enforce the bounds \eqref{eq:kkt_ineq_feas} and \eqref{eq:kkt_mult_bound}.  That said, Newton's method can be applied indirectly.  For example, interior-point methods deal with the bounds \eqref{eq:kkt_ineq_feas} and \eqref{eq:kkt_mult_bound} by introducing a homotopy map with a barrier term and using a sequence of Newton solves, while active-set methods attempt to predict the active inequality constraints and treat them as equality constraints during a Newton-like iteration.

Our approach for dealing with inequality constraints is related to the active-set approach and is based on the following theorem due to Mangasarin~\cite{Mangasarian1976equivalence}.

\begin{thrm}\label{thrm:mangasarian}
Let $G:\mathbb{R}\rightarrow\mathbb{R}$ be any strictly increasing function, that is $a > b \Leftrightarrow G(a) > G(b)$, and let $G(0) = 0$.  Then $\ineq_i(x^*)$ and $\lagi_i^*$ satisfy the complementarity conditions \eqref{eq:kkt_comp}, \eqref{eq:kkt_ineq_feas}, and \eqref{eq:kkt_mult_bound} if and only if
\begin{equation*}
	G(|\ineq_i(x^*) - \lagi_i^*|) - G(\ineq_i(x^*)) - G(\lagi_i^*) = 0.
\end{equation*}
\end{thrm}

\begin{proof}
The original proof given in \cite{Mangasarian1976equivalence} holds with $\ineq_i(x^*)$ taking the role of $F_i(z)$ and $\lagi_i$ taking the role of $z_i$.  \qed
\end{proof}

Theorem~\ref{thrm:mangasarian} shows that we can replace the problematic complementarity conditions with an equivalent set of nonlinear equations.  Furthermore, these nonlinear equations will be amenable to Newton-like solution methods if we make an appropriate choice for the function $G$ appearing in the theorem.  Here we adopt the simple choice $G(z) = \frac{1}{2} z$.  This choice leads to a nonlinear system that is differentiable almost everywhere and, therefore, is suitable for Newton-like methods.

\begin{remark}
The factor of $\frac{1}{2}$ in $G(z) = \frac{1}{2} z$ is not strictly necessary; it merely avoids the factor of 2 in the nonlinear equation $|\ineq_i(x) - \lagi_i| - \ineq_i(x) - \lagi_i$, which evaluates to $2\ineq_i(x)$ or $2\lagi_i$, depending on the sign of $\ineq_i(x) - \lagi_i$.
\end{remark}

\begin{remark}
With $G(z) = \frac{1}{2} z$, the nonlinear equation $\frac{1}{2}\left( |\ineq_i(x) - \lagi_i| - \ineq_i(x) - \lagi_i \right)$ is not differentiable along $\ineq_i(x) = \lagi_i$.  However, if we have strict complementarity at the solution, \ie $\ineq_i(x^*) - \lagi^* \neq 0$, the equation is locally differentiable as required by Newton's method.  Thus, our approach is in the class of semi-smooth Newton methods~\cite{Qi1993nonsmooth}.
\end{remark}

We conclude this section by summarizing the above results and introducing some notation and definitions to make the subsequent presentation more concise.  Let $y \in \mathbb{R}^{N}$, with $N \equiv n+\me+\mi$, be the compound vector composed of the primal variables and multipliers:
\begin{equation*}
  y \equiv \begin{bmatrix} x^T & \lage^T & \lagi^T \end{bmatrix}^T.
\end{equation*} 
In addition, let the nonlinear residual function $r : \mathbb{R}^{N} \rightarrow \mathbb{R}^{N}$ be defined by
\begin{equation}
  r(y) \equiv
  \begin{bmatrix}
  \nabla L(x,\lage,\lagi) \\
    -\eq(x) \\
    \frac{1}{2} \left( |\ineq(x) - \lagi | - \ineq(x)  - \lagi \right)
\end{bmatrix}, \label{eq:res} \\
\end{equation}
where the Lagrangian, $L(x,\lage,\lagi)$ was defined in Theorem~\ref{thrm:kkt}, and the absolute value in the last block is to be interpreted componentwise.  Using these definitions, Theorems~\ref{thrm:kkt} and \ref{thrm:mangasarian} imply the following result.

\begin{corollary}
Let $x^*$ be a local solution of \eqref{eq:opt}, and assume that active constraint gradients are linearly independent.  Then there exists multipliers $\lage^*$ and $\lagi^*$ such that $y^* = \begin{bmatrix} x^{*T} & \lage^{*T} & \lagi^{*T} \end{bmatrix}^{T}$ satisfies
\begin{equation*}
	r(y^*) = 0.
\end{equation*}
\end{corollary}

Our basic approach to finding a local solution to the optimization problem \eqref{eq:opt} is to solve $r(y)=0$.  The following section describes the multisecant algorithm that we use for this purpose. 


\section{Algorithm Description}\label{sec:algorithm}

\subsection{Newton's Method and Mutisecant Methods}

Before describing our particular algorithm, we briefly review the class of
multisecant methods upon which it is based.  For a more complete review of
multisecant methods see~\cite{Fang2009two}.

In this subsection, we consider the generic problem of solving $r(y) =
0$, where $r : \mathbb{R}^{N} \rightarrow \mathbb{R}^{N}$ is
continuously differentiable almost everywhere.  In particular, we will not be concerned with avoiding stationary points that are not local minimizers of \eqref{eq:opt}; we will address nonconvexity in \ref{sec:nonconvex}.  

Multisecant methods are approximations to Newton's method.  Newton's method
itself is based on the linear approximation\footnote{Subscripts in this section refer to iteration number, not the component of the vector as they did in the previous section.}
\begin{equation}\label{eq:linear_approx}
  r(y_k + \Delta y_k) \approx r(y_k) + \mat{J}(y_k) \Delta y_k,
\end{equation}
where $y_{k}$ is the estimated solution at iteration $k$, and $\mat{J} :
\mathbb{R}^{N} \rightarrow \mathbb{R}^{N\times N}$ is the Jacobian $\partial r/\partial y$.
The Newton step is found by setting the right-hand side above to zero and
solving for $\Delta y_k$:
\begin{equation}\label{eq:newton}
  \Delta y_k = -\mat{J}(y_k)^{-1} r_k
\end{equation}
where $r_k \equiv r(y_k)$.  The next iterate of Newton's method is then obtained as $y_{k+1} = y_k + \Delta y_{k}$, usually with some safe-guards on the step to ensure globalization.

A potential disadvantage of Newton's method is that forming and/or inverting the
Jacobian can be expensive.  There are several strategies for reducing or
avoiding these costs, one being the inexact-Newton-Krylov class of methods~\cite{Knoll2004jacobian}; Newton-Krylov methods require only Jacobian-vector products, and therefore avoid
the need to form the Jacobian explicitly.  Quasi-Newton methods offer an
alternative strategy that avoids the need to form the exact Jacobian; they store
an approximation to the Jacobian (or its inverse), $\mat{J}_k \approx
\mat{J}(\bm{y}_k)$, and update this approximation at each iteration using
low-rank matrices.  Multisecant methods belong to the class of quasi-Newton
methods.

Multisecant methods get their name from the secant condition, which is obtained
by replacing the approximation in \eqref{eq:linear_approx} with an equality:
\begin{equation}\label{eq:sec}
  \mat{J}_{k+1} \Delta y_k = \Delta r_k,
\end{equation}
where $\Delta r_k \equiv r_{k+1} - r_k$.  Alternatively, when
approximating the inverse Jacobian with $\mat{G}_{k+1} \approx
\mat{J}(\bm{y}_{k+1})^{-1}$, the secant condition becomes
\begin{equation}\label{eq:sec_inv}
  \mat{G}_{k+1} \Delta r_k = \Delta y_k.
\end{equation}
The secant condition is the basis for Broyden's method for nonlinear
equations~\cite{Broyden1965class}, as well as several popular quasi-Newton methods for
optimization, namely DFP~\cite{Davidon1991variable}, BFGS~\cite{Broyden1970convergence,Fletcher1970new,Goldfarb1970new,Shanno1970conditioning}, and SR1~\cite{Conn1991convergence}.

Rather than a single secant condition, Vanderbilt and Louie~\cite{Vanderbilt1984total} and
Eyert~\cite{Eyert1996comparative} proposed generalizations of Broyden's method that require
$\mat{G}_{k}$ (or $\mat{J}_{k}$) to satisfy a set of $q$ secant equations.  If we define
\begin{align*}
  \mat{Y}_{k} &= \begin{bmatrix}
    \Delta y_{k-q} & \Delta y_{k-q+1} & \cdots & \Delta y_{k-1}
  \end{bmatrix} \\
  \text{and}\qquad
  \mat{R}_{k} &= \begin{bmatrix}
    \Delta r_{k-q} & \Delta r_{k-q+1} & \cdots & \Delta r_{k-1}
  \end{bmatrix},\qquad
\end{align*}
then the $q$ previous secant conditions can be written succinctly as
\begin{equation}\label{eq:multisec}
  \mat{G}_{k} \mat{R}_k = \mat{Y}_k.
\end{equation}
In general $q < N$, so the $q$ conditions in \eqref{eq:multisec} are
insufficient to define $\mat{G}_k$, and additional conditions are necessary.

In the context of optimization, the Jacobian of the KKT conditions is symmetric, so symmetry provides another condition we might consider imposing on $\mat{G}_k$.  Indeed, most quasi-Newton methods for optimization that are based on the single secant condition, \eqref{eq:sec} or
\eqref{eq:sec_inv}, do impose symmetry on the approximation of the Hessian, KKT
matrix, or their inverses.  Therefore, it is interesting to consider whether
symmetry can be imposed on a multisecant quasi-Newton approximation.
Unfortunately, it is easy to show that the answer is negative.

\begin{prop}\label{thrm:sym}
  A quasi-Newton approximation $\mat{G}_{k}$ that satisfies the multisecant
  condition~\eqref{eq:multisec} cannot be symmetric, in general, if $q > 1$.
\end{prop}

\begin{proof}
  Assume that $\mat{G}_{k} = \mat{G}_{k}^T$.  Then, left multiplying
  \eqref{eq:multisec} by $\mat{R}_k^{T}$, we have
  \begin{equation*}
    \mat{R}_k^T \mat{Y}_k = \mat{R}_k^{T} \mat{G}_{k} \mat{R}_k =
    \mat{R}_k^{T} \mat{G}_{k}^T \mat{R}_k
    = \mat{Y}_k^T \mat{R}_k.
  \end{equation*}
  This implies that $\mat{R}_k^T \mat{Y}_k$ is symmetric, which is not possible in general.     Therefore, our assumption on the possible symmetry of $\mat{G}_k$ must be false. \qed
\end{proof}

\begin{remark}
While the product $\mat{R}_k^T \mat{Y}_k$ is not symmetric in general, there are at least two cases for which it is symmetric.  The first case is when $q=1$ (one secant condition), and the second case is when the objective is quadratic and constraints are linear; in the latter case, we have $\mat{R}_k = \mat{B} \mat{Y}_k$ for some symmetric matrix $\mat{B}$.
\end{remark}

While we cannot impose symmetry on $\mat{G}_k$, we can follow the approach used
in the generalized Broyden's method~\cite{Eyert1996comparative}; specifically, $\mat{G}_{k}$ is
required to be as close as possible, in the Frobenius norm, to some
estimate of the inverse Jacobian, which we will denote by $\tilde{\mat{G}}_k$.  For example, 
\cite{Eyert1996comparative} proposes using the previous estimate for the inverse Jacobian, that is $\tilde{\mat{G}}_k = \mat{G}_{k-q}$.  We will discuss other possible choices for $\tilde{\mat{G}}_k$ in the next subsection.

The requirement that $\mat{G}_k$ be as close as possible to $\tilde{\mat{G}}_k$,
together with the secant condtions~\eqref{eq:multisec}, produces the closed-form
expression for $\mat{G}_k$ provided in the following theorem.

\begin{thrm}\label{thrm:Broyden}
  The approximate inverse of the Jacobian at iteration $k$ of the generalized
  Broyden's method is
  \begin{equation}\label{eq:broyden}
    \mat{G}_k = \tilde{\mat{G}}_k + (\mat{Y}_k - \tilde{\mat{G}}_k \mat{R}_k)(\mat{R}_k^T
    \mat{R}_k)^{-1} \mat{R}_k^T,
  \end{equation}
  which is the solution of
  \begin{equation*}
    \min_{\mat{G} \in \mathbb{R}^{N\times N}} \quad \| \mat{G} - \tilde{\mat{G}}_k \|_{F},
    \qquad\text{s.t} \quad \mat{G} \mat{R}_k = \mat{Y}_k,
  \end{equation*}
  where $\| \cdot \|_{F}$ denotes the Frobenius norm.  
\end{thrm}

\begin{proof} The proof of this result can be found in~\cite{Eyert1996comparative} and \cite{Fang2009two}.\qed
\end{proof}

To summarize, a multisecant method makes the approximation $\mat{G}_k \approx \mat{J}(y_k)^{-1}$ in the Newton update \eqref{eq:newton}, where $\mat{G}_k$ is defined by \eqref{eq:broyden}.  Thus, the next iterate in a multisecant method is given by
\begin{equation}\label{eq:ms_update}
  y_{k+1} = y_{k} - \left[\tilde{\mat{G}}_k + (\mat{Y}_k - \tilde{\mat{G}}_k \mat{R}_k)(\mat{R}_k^T
    \mat{R}_k)^{-1} \mat{R}_k^T \right] r_k.
\end{equation}

\subsection{Choosing $\tilde{\mat{G}}_k$}

In order to complete the definition of the multisecant update we must specify a choice for $\tilde{\mat{G}}_k$.  The simplest choice for $\tilde{\mat{G}}_k$ is the scaled identity:
\begin{equation}
	\tilde{\mat{G}}_k = \alpha \mat{I}, 
\end{equation}
where $\alpha > 0$.  This choice makes the update~\eqref{eq:ms_update} equivalent to Anderson-mixing~\cite{Eyert1996comparative}.  Furthermore, in the case of unconstrained optimization with no secant conditions on $\mat{G}_k$ (\ie $q = 0$), the multisecant update reduces to steepest descent, and the parameter $\alpha$ determines the step length via $\| \Delta y_k \| = \alpha \| r(y_k) \|$. 

Even when secant conditions are imposed on $\mat{G}_k$ ($q > 0$), $\alpha$ can (and should) be chosen to influence the step lengths, preventing overly conservative or aggressive updates; however, when $q > 0$, there is no guarantee that the step at iteration $k$ will have length $\alpha \| r(y_k) \|$.  This is because $\mat{G}_k$ must respect the secant conditions, whereas $\| \mat{G}_k - \alpha \mat{I} \|$ is only minimized; see Theorem~\ref{thrm:Broyden}.  This is potentially beneficially, because it is possible for the algorithm to overcome a poor choice for $\alpha$ if $q$ is sufficiently large.

We have found that using the scaled identity for $\tilde{\mat{G}}_k$ is adequate for reasonably well-scaled unconstrained problems, as well as constrained problems with modest numbers of active constraints (fewer than 10); however, for ill-conditioned problems, a more sophisticated choice is necessary.  Ideally, we would use $\tilde{\mat{G}}_k = \mat{J}(y_k)^{-1}$.  While this choice is not practical, it does suggest a class of options for $\tilde{\mat{G}}_k$: preconditioners.

Consider the iterative solution of the Newton update \eqref{eq:newton} using a Krylov method.  For this popular class of methods, the number of iterations is related to the condition number of the linear system~\cite{Saad2003iterative}, so preconditioners are employed to cluster the eigenvalues and/or reduce the conditioning of the system.  Preconditioners are usually based on approximations to
 $\mat{J}(y_k)^{-1}$, which is precisely what we want for $\tilde{\mat{G}}_k$.
Thus, any suitable preconditioner designed for the Krylov-iterative solution of \eqref{eq:newton} can be adopted for $\tilde{\mat{G}}_k$.  We will explore this possibility in the numerical results.

\begin{remark}
  The connection between Krylov and multisecant methods goes beyond their mutual need for preconditioning.  For instance, Walker and Ni~\cite{Walker2011anderson} have shown that the iterates of Anderson acceleration can be obtained from those of GMRES~\cite{Saad1986gmres}, and vice versa, when Anderson acceleration is used to solve linear systems.
\end{remark}

\begin{remark}
 $\tilde{\mat{G}}_{k}$ does not need to be the same matrix at every iteration of the multisecant method, that is, nonstationary preconditioners are permitted.  In this regard, multisecant methods are similar to flexible Krylov iterate methods like FGMRES~\cite{Saad1993flexible}.
\end{remark}

As in the simple case $\tilde{\mat{G}}_k = \alpha \mat{I}$, scaling the chosen preconditioner by $\alpha$ can help control the step size in the early stages of the multisecant method.  If we let $\mat{P}_k^{-1}$ denote a generic preconditioner, then the choice
\begin{equation*}
	\tilde{\mat{G}}_k = \alpha \mat{P}_k^{-1},
\end{equation*}
with $\alpha > 0$, encapsulates both the scaled diagonal ($\mat{P}_k^{-1} = \mat{I}$) as well as more elaborate preconditioners.  This is the form for $\tilde{\mat{G}}_k$ we will use throughout the remaining paper.

\subsection{Handling Nonconvex Problems}\label{sec:nonconvex}

Recall that the KKT conditions~\eqref{eq:kkt}, as well as the equivalent conditions $r(y)=0$, are \emph{necessary}, but not sufficient, conditions for a solution to \eqref{eq:opt}.  Other stationary points, including local maximizers and saddle points, also satisfy $r(y) = 0$.  Newton's method has no way to distinguish between these different types of stationary points, and it can easily converge to the wrong type.  A multisecant method, being based on Newton's method, will suffer the same fate if it is not safe-guarded.

Before we describe how we handle nonconvexity, we first review common strategies used by existing algorithms, and explain why these are not suitable for a multisecant method.
\begin{itemize}
\item Quasi-Newton methods like BFGS are updated in such a way that the approximate inverse of the Lagrangian Hessian remains positive definite.  Consequently, they are guaranteed to produce a descent direction.  Unfortunately, multisecant methods cannot produce a symmetric Hessian, in general, let alone one that is positive definite; see Theorem~\ref{thrm:sym}.

\item For unconstrained problems, a step direction can easily be checked to see if it is a descent direction.  If the step is not a descent direction, it can be discarded and we can resort to a steepest-descent step, for example.  This approach is viable for our algorithm, but it is limited to unconstrained problems.

\item Many optimization algorithms for constrained problems require the user to provide the constraint Jacobian, which can then be factored to determine a basis for its null-space.  Using this basis, an algorithm can project the problem onto a reduced-space, effectively turning it into a unconstrained problem.  In this reduced space, a step can be checked to see if it is a descent direction, analogous to the unconstrained case.  This approach is not possible for multisecant methods, because the Jacobian is not explicitly available.
\end{itemize}

Instead of the above methods, we use a simple Hessian-regularization approach to address nonconvexity.  In the context of the Newton update \eqref{eq:newton}, undesirable steps caused by indefinite Hessians\footnote{Specifically, Hessians that are not positive-definite in the null-space of the linearized constraints.} can be prevented by adding a scaled identity, $\beta \mat{I}$, to the Hessian of the Lagrangian, provided $\beta > 0$ is larger than the most negative eigenvalue of the projected Hessian.
 
While we do not have direct access to the Hessian of the Lagrangian in the multisecant update, the effect of adding $\beta \mat{I}$ to the Hessian can be mimicked by modifying the difference vectors $\Delta r_k$ as follows:
\begin{equation*}
  \Delta r_k \equiv r(y_{k+1}) - r(y_k) + \begin{bmatrix} \beta(x_{k+1} - x_{k}) \\ 0 \\ 0 \end{bmatrix}.
\end{equation*}

There are two drawbacks to Hessian regularization.  First, it will limit the asymptotic rate of convergence to linear, and, second, the ideal value for $\beta$ requires an estimate of the negative eigenvalue of greatest magnitude.  In our experience, the impact on the asymptotic rate of convergence is of minor practical concern: superlinear convergence of the unregularized method is typically limited to the last two or three iterates.

The estimate of $\beta$ is a more serious concern.  In this work we have used trial and error to determine a suitable value for $\beta$.  A more methodical approach would be to use a few iterates of the Lanzcos method~\cite{Lanczos1950iteration} applied to the Hessian of $L$ after convergence, in order to estimate the negative eigenvalues of largest magnitude, if any.  The Lanzcos method is attractive here, because it requires only Hessian-vector products, which can be approximated in a matrix-free manner using a forward difference applied to the Lagrangian gradient ($\nabla L$ is already required by our algorithm).  If negative eigenvalues are found, the algorithm can be restarted with the value of $\beta$ set appropriately.  However, this posterior Lanzcos approach may be overly conservative, since the only negative eigenvalues of significance are those in the null-space of the constraint Jacobian.

\subsection{The Multisecant Accelerated Descent (MAD) Algorithm}

Our proposed optimization method, Multisecant Accelerated Descent (MAD), is summarized in Algorithm~\ref{alg:MAD}.  There are several implementation details that are important to highlight.
\begin{itemize}

\item We assume that the initial multiplier values are zero; see Line~\ref{line:init}.

\item There are different criteria that can be used to assess convergence of the first-order optimality conditions in Line~\ref{line:converge}.  In our implementation, we accept the solution if relative and absolute tolerances on primal optimality and feasibility are met, specifically
\begin{equation}\label{eq:criterion}
\begin{gathered}
	\| \nabla f(x_k) - (\nabla \eq(x_k)^T)\lage_k - (\nabla \ineq(x_k)^T)\lagi_k \| 
	\leq \epsilon_r \| \nabla f(x_0) \| + \epsilon_a \\
	\left\| \begin{matrix} \eq(x_k) \\ \frac{1}{2}(|\ineq(x_k) - \lage_k | - g(x_k) - \lage_k \end{matrix} \right\| \leq \epsilon_r  \left\| \begin{matrix} \eq(x_0) \\ \frac{1}{2}(|\ineq(x_0) - \lage_0 | - g(x_0) - \lage_0 \end{matrix} \right\| + \epsilon_a,
\end{gathered}
\end{equation}
where $\epsilon_r \in (0,1)$ and $\epsilon_a > 0$ are relative and absolute tolerances, respectively.  We use the same tolerances for both primal optimality and feasibility in this work, because our problems are relatively well scaled; in general, different tolerances may be necessary for the two criteria.

\item For problems with noisy/inaccurate data, the computational budget defined by the maximum number of iterations, $K_{\max}$, may be exceeded before the convergence criteria in Line~\ref{line:converge} are met.

\item The least-squares subproblem on Line~\ref{line:least_squares} corresponds to the vector 
\begin{equation*}
	(\mat{R}_k^T \mat{R}_k)^{-1} \mat{R}_k^T r_k = \gamma,
\end{equation*}
seen in the multisecant update~\eqref{eq:ms_update}.  The above expression is the normal-equation solution to the overdetermined problem $\mat{R}_k \gamma = r_k$.  While the normal-equation solution is convenient theoretically, it is not advisable in practice~\cite{Fang2009two}.  This is because the columns of $\mat{R}_k$ can become close to linearly dependent, leading to ill-conditioning in $\mat{R}_k^T \mat{R}_k$.  A common solution to this possible ill-conditioning, and the one adopted here, is to use a truncated singular-value decomposition.  In particular, we truncate singular values that are smaller than $10^{-6}$ relative to the largest singular value.

\item After evaluating the full multisecant step in Line~\ref{line:ms_update}, we check its magnitude in Line~\ref{line:safe_guard} and limit the step to a maximum length of $\Delta_{\max}$, if necessary.

\end{itemize}

\begin{algorithm}[htbp]\DontPrintSemicolon
  \KwData{$x_{0}$, $\alpha > 0$, $\beta > 0$, $\Delta_{\max}$, $q \geq 0$, $K_{\max} \geq 0$, and operator $\mat{P}_k^{-1}$}
  \KwResult{$y$, a approximate solution to $r(y) = 0$}
  \BlankLine
  set $y_0 = \begin{bmatrix} x_0^T & 0^T & 0^T \end{bmatrix}^T$ and 
  compute and store $r_0 = r(y_0)$ \label{line:init}\; 
  \For{$k = 0,1,2,\ldots,K_{\max}$}{
  	\If{$\| r_k \|$ is sufficiently small\label{line:converge}}{
		return $y_k$\;
	}
	\For{$j = 1,2,\ldots,\min(k,q)$}{
		$\Delta y_{k-j} \leftarrow y_{k-j+1} - y_{k-j}$\;
		$\Delta r_{k-j} \leftarrow r_{k-j+1} - r_{k-j} + \beta \begin{bmatrix} \Delta x_{k-j} \\ 0 \\ 0 \end{bmatrix}$\;
	}
	solve $\argmin_{\gamma} \| r_k - \mat{R}_k \gamma \|$ \label{line:least_squares}\;
	$\Delta y_k = - \alpha \mat{P}_k^{-1} r_k - (\mat{Y}_k - \alpha \mat{P}_k^{-1} \mat{R}_k)\gamma$ \label{line:ms_update}\;
	\If{$\| \Delta y_k \| > \Delta_{\max}$ \label{line:safe_guard}}{
		$\Delta y_k \leftarrow \frac{\Delta_{\max}}{\| \Delta y_k \|} \Delta y_k$\;
	}
	$y_{k+1} = y_k + \Delta y_k$\;
	$r_{k+1} = r(y_{k+1})$\;
  }
  \caption{Multisecant Accelerated Descent.\label{alg:MAD}}
\end{algorithm}

Other than the rudimentary step-length safeguard in Line~\ref{line:safe_guard}, our algorithm has no globalization strategies, such as line-search or trust-region methods.  This is unusual and demands some justification.  

When the application has noisy or inaccurate data\footnote{Data here refers to $r(y)$ and $f(x)$.}, it is difficult to distinguish a good step from a poor step.  Consider an unconstrained optimization algorithm in which a sufficient-decrease line search is implemented, and suppose a step causes the objective function to violate the sufficient-decrease condition.  Did this violation happen because the step was poor and the ``true'' value of $f$ increased, or did this happen because of inaccurate data?  If we know enough about the nature of the error in the data, we may be able to answer this question; however, in general, the behavior of the error will be unknown and globalization methods will be unreliable.

When the data is accurate, our justification is more pragmatic: the method seems to more efficient without globalization.  One possible explanation is that multisecant methods may, sometimes, benefit from ``bad" steps, since these steps provide information about the curvature of the problem.  Furthermore, recall that $\alpha$ does provide some control over the step-length size, so it acts as a kind of implicit globalization.  That said, a more rigorous and efficient globalization for error-free problems may be possible with further analysis and investigation.

\begin{remark}
We are not the first to forego standard globalization techniques when using a multisecant method.  Fang and Saad~\cite{Fang2009two} also avoid explicit globalization when using multisecant methods to solve electronic structure calculations.  Their motivation is to minimize unnecessary evaluations of the self-consistent field iteration, which is expensive.  Our target application, PDE-constrained optimizaiton, also leads to expensive function evaluations, so we are also sensitive to the additional evaluations required by globalization methods.
\end{remark}

\section{Numerical Experiments}\label{sec:results}

\subsection{Multidisciplinary Design Optimization Problem}

In the following sets of experiments, we consider a model multidisciplinary design optimization (MDO) problem.  The problem consists of finding a nozzle geometry such that the quasi-one-dimensional Euler equations produce a pressure that is as close as possible to a target pressure.  The nozzle itself can (statically) deform under the pressure loading, and this deformation is modeled using a one-dimensional finite-element beam.  The aerostructural optimization problem is described in detail in Reference~\cite{Dener2017matrix}; below we provide a brief description as needed for the present study.

The MDO problem is posed using the individual-discipline feasible (IDF) formulation~\cite{Haftka1992options,Cramer1994problem}.  The IDF formulation introduces additional optimization variables, called coupling variables, that allow the disciplinary state equations to be solved independently at each optimization iteration.  This helps maintain modularity and also avoids coupled multidisciplinary analyses and coupled adjoints; however, the IDF formulation makes the optimization problem more difficult, because there are many more variables, and it introduces state-based constraints whose Jacobian is expensive to evaluate.  This makes matrix-free optimization methods attractive for this type of problem.

For the elastic-nozzle problem, the IDF optimization statement is
\begin{equation}\label{eq:idf}\tag{\text{IDF}}
  \begin{alignedat}{2}
    &\underset{b,\bar{p},\bar{u}_y}{\mathsf{minimize}} & f(b,\bar{p},\bar{u}_y)&, \\
    &\mathrm{s.t.} & p(b,\bar{u}_y) - \bar{p} &= 0, \\
    &              &  u_y(b,\bar{p}) - \bar{u}_y &= 0.
  \end{alignedat}
\end{equation}
The objective function is a discretization of $\int (p - p_{t})^2\, dx$, where $p$ is the pressure and $p_{t}$ is the target pressure; again, see~\cite{Dener2017matrix} for the details.  The optimization variables consist of 1) the b-spline control points, $b$, that define the unloaded shape of the nozzle, 2) the pressure coupling variables, $\bar{p}$, that are used by the structural model to define the loading, and 3) the nozzle vertical-displacement coupling variables, $\bar{u}_y$, that are used by the flow model to define the nozzle shape (static + displacement).  For a valid solution to \eqref{eq:idf}, the coupling variables $\bar{p}$ and $\bar{u}_y$ must agree with the values of pressure and vertical displacement, respectively, predicted by the analyses.  This requirement is expressed by the (vector) equality constraints in \eqref{eq:idf}.

In order to evaluate the objective and constraints in \eqref{eq:idf} at the $k$th optimization iteration, we must first solve the disciplinary state equations based on the given values of $b_k$, $\bar{p}_k$, and $(\bar{u}_y)_k$.  Furthermore, gradients of the objective and constraints require the solution of adjoint equations. In this work, both the state and adjoint equations are solved iteratively, the former with a Newton solver and the latter with a preconditioned Krylov method.  We will use the iterative solvers' tolerances to control the accuracy of the state and adjoint solutions when we study the impact of inaccurate data on our algorithm.

We benchmark the MAD solution of \eqref{eq:idf} against a previously developed inexact-Newton-Krylov algorithm~\cite{Hicken2015flecs,Dener2016kona}.  Dener and Hicken~\cite{Dener2014revisiting,Dener2017matrix} recently developed a specialized preconditioner that takes advantage of the structure in the IDF formulation.  This preconditioner is used in Algorithm~\ref{alg:MAD} for all solutions of \eqref{eq:idf}, unless stated otherwise.

We use 5 b-spline control points and the mesh has 31 nodes, so $\bar{p}, \bar{u}_y \in \mathbb{R}^{31}$ and $\me = 2\times 31 = 62$; including the multipliers, this gives a problem with $N = 5 + 4\times 31 = 129$ variables.  For all experiments below, the initial guess for $x_0$ is the same used in \cite{Dener2017matrix}.  The relative and absolute tolerances are $\epsilon_r = 10^{-4}$ and $\epsilon_a = 10^{-6}$, respectively.  The maximum allowable primal step is $\Delta_{\max} = 1.0$, and the maximum number of iterations is $K_{\max} = 2000$.

\subsubsection{Parameter Study}

We begin by investigating, in the context of \eqref{eq:idf}, the effect of varying the primary parameters in Algorithm~\ref{alg:MAD}.  These parameters are 1) the number of saved vectors\footnote{Alternatively, $q-1$ is the number of columns in $\mat{R}_k$ and $\mat{Y}_k$, once $k \geq q$.}, $q$, 2) the preconditioner scaling parameter, $\alpha$, and 3) the Hessian regularization parameter $\beta$.

Table~\ref{tab:idf_alpha} lists the number of iterations used by MAD for a range of $q$ and $\alpha$ values, with $\beta = 0.5$.  As $q$ increases for fixed $\alpha$, the number of iterations generally decreases; the only exception to this is $q=10$ and $\alpha=0.01$, which did not converge in $K_{\max} = 2000$ iterations.  Intuitively, increasing the number of secant conditions that are satisfied should increase the accuracy of the approximate inverse $\mat{G}_k$ and, therefore, reduce the number of iterations, at least when sufficiently close to the solution that the linear approximation to $r(y)$ is accurate.

Increasing $\alpha$ also improves performance, up to a point, although the trends are less consistent. Recall that $\alpha$ influences the step length, particularly during the first few iterations.  Consequently, relatively large $\alpha$ may lead to aggressive steps; indeed, increasing $\alpha$ beyond $\alpha=0.1$ leads MAD to diverge on problem \eqref{eq:idf}.  Conversely, small values of $\alpha$ can lead to conservative steps and many iterations.  These general trends are reflected in the data.

\begin{table}[tbp]
\centering
\caption{Number of iterations required by MAD on problem \eqref{eq:idf} with accurate data, for different combinations of $q$ and $\alpha$.  The Hessian regularization parameter is fixed at $\beta=0.5$.  An ``$\infty$" denotes a run that did not converge in fewer than $K_{\max}$ iterations, and ``nan'' denotes a run that diverged.}
\label{tab:idf_alpha}
\begin{tabular}{l|rrrrr}
\multicolumn{1}{c}{} & \multicolumn{5}{c}{$q$} \\\cline{2-6}
\rule{0ex}{3ex} $\alpha$ & \multicolumn{1}{c}{5} & \multicolumn{1}{c}{10} & 
\multicolumn{1}{c}{15} & \multicolumn{1}{c}{20} & \multicolumn{1}{c}{25} \\\hline
\rule{0ex}{3ex}%
0.01 & 1773 & $\infty$ &  432 & 156  &  73  \\
0.05 &  600  &      562  &  254 & 231 & 263 \\
0.1   &  277  &       78   &    42 &  35  &   35 \\
1.0   & nan   &  nan      & nan  &  nan&  nan \\\hline
\end{tabular}
\end{table}

\ignore{
\begin{table}[tbp]
\centering
\caption{Number of preconditioner applications required by MAD on problem \eqref{eq:idf} with accurate data.}
\label{tab:idf_accurate}
\begin{tabular}{lrrrrr}
 & \multicolumn{5}{c}{$q$} \\\cline{2-6}
\rule{0ex}{3ex} $\alpha$ & \multicolumn{1}{c}{5} & \multicolumn{1}{c}{10} & 
\multicolumn{1}{c}{15} & \multicolumn{1}{c}{20} & \multicolumn{1}{c}{25} \\\hline
\rule{0ex}{3ex}%
0.01 & 92107 & $\infty$ & 22738 &  8261 &   3984 \\
0.05 & 30950 & 28784  & 13290 & 12059 & 13701 \\
0.1   & 14368 &  4038   &  2218  &  1844  &  1842  \\\hline
\end{tabular}
\end{table}
}

The results of varying $\beta$ and $q$ are listed in Table~\ref{tab:idf_beta}.  For all runs in the table, $\alpha$ was held fixed at $0.1$.  As before, we see improved performance as $q$ increases.  However, the result of increasing $\beta$ is not as expected.  Increasing the Hessian regularization should reduce the effective step size, increasing the number of iterations and improving robustness.  Instead, the method becomes unstable and diverges as $\beta$ is increased beyond a value of one.  For the values considered, $\beta = 0.5$ appears to be optimal, but it is not clear why.  Further investigation into the role of $\beta$ is necessary.

\begin{table}[tbp]
\centering
\caption{Number of iterations required by MAD on problem \eqref{eq:idf} with accurate data, for different combinations of $q$ and $\beta$.  The scaling parameter is fixed at $\alpha=0.1$.  The symbol ``nan'' denotes a run that diverged.}
\label{tab:idf_beta}
\begin{tabular}{l|rrrrr}
\multicolumn{1}{c}{} & \multicolumn{5}{c}{$q$} \\\cline{2-6}
\rule{0ex}{3ex} $\beta$ & \multicolumn{1}{c}{5} & \multicolumn{1}{c}{10} & 
\multicolumn{1}{c}{15} & \multicolumn{1}{c}{20} & \multicolumn{1}{c}{25} \\\hline
\rule{0ex}{3ex}%
0.0  & 224 & 571 & 111  & 45  & 50   \\
0.1  & 329 & 254 & 156  & 80 & 138 \\
0.5  & 277 & 78   & 42    & 35 & 35  \\
1.0  & 354 & 414 & 179 & 114 & 70 \\
5.0  & nan & nan & nan & nan & nan \\\hline
\end{tabular}
\end{table}

\subsubsection{Performance Using Accurate Data}

Figure~\ref{fig:IDF_accurate} shows the convergence histories of the MAD algorithm applied to the multidisciplinary design problem~\eqref{eq:idf}.  Specifically, the plots show the optimality and feasibility norms on the left side of \eqref{eq:criterion}, normalized by their initial values.  Histories are plotted for a range of $q$ values from $5$ to $25$, with fixed values of $\alpha=0.1$ and $\beta=0.5$.  The convergence history produced by the Newton-Krylov (NK) algorithm~\cite{Dener2017matrix} is included for comparison.  The abcissa is the computational cost normalized by the cost of the NK method.

The results in Figure~\ref{fig:IDF_accurate} are based on tightly converged state and adjoint residuals with relative tolerances of $10^{-10}$ and $10^{-6}$, respectively.  Consequently, the data in this case is sufficiently accurate that conventional optimization methods, like the NK algorithm, will not experience issues with globalization.  

The results show that, for this particular case with accurate data, the MAD algorithm is competitive with the NK algorithm, especially for values of $q$ larger than 15.  Indeed, for $q \geq 15$ the asymptotic convergence rate appears similar to that of the inexact-Newton method, \ie superlinear.

\begin{figure}[tbp]
  \subfigure[relative optimality \label{fig:opt_IDF_accurate}]{%
    \includegraphics[width=0.48\textwidth]{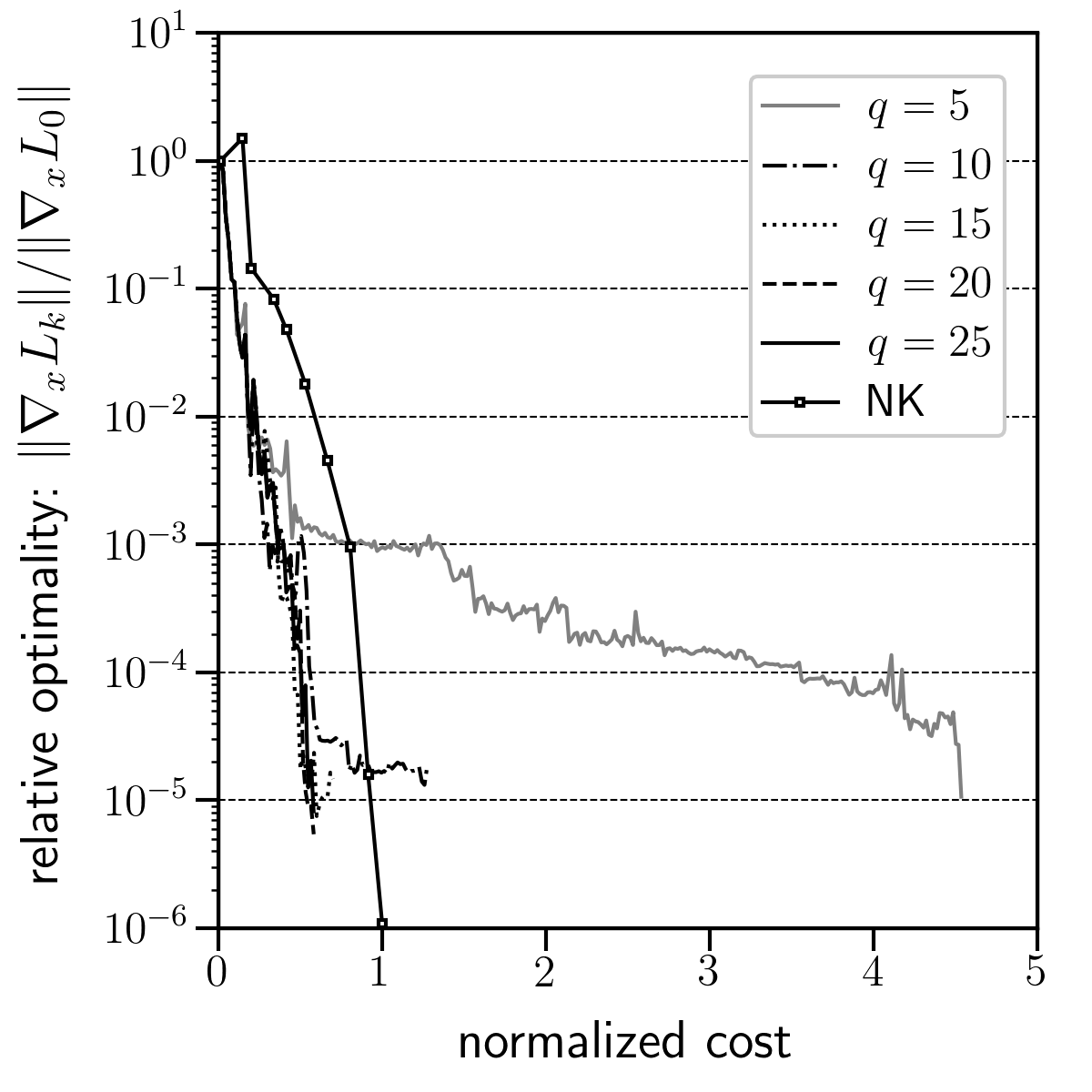}}
  \subfigure[relative feasibility \label{fig:feas_IDF_accurate}]{%
    \includegraphics[width=0.48\textwidth]{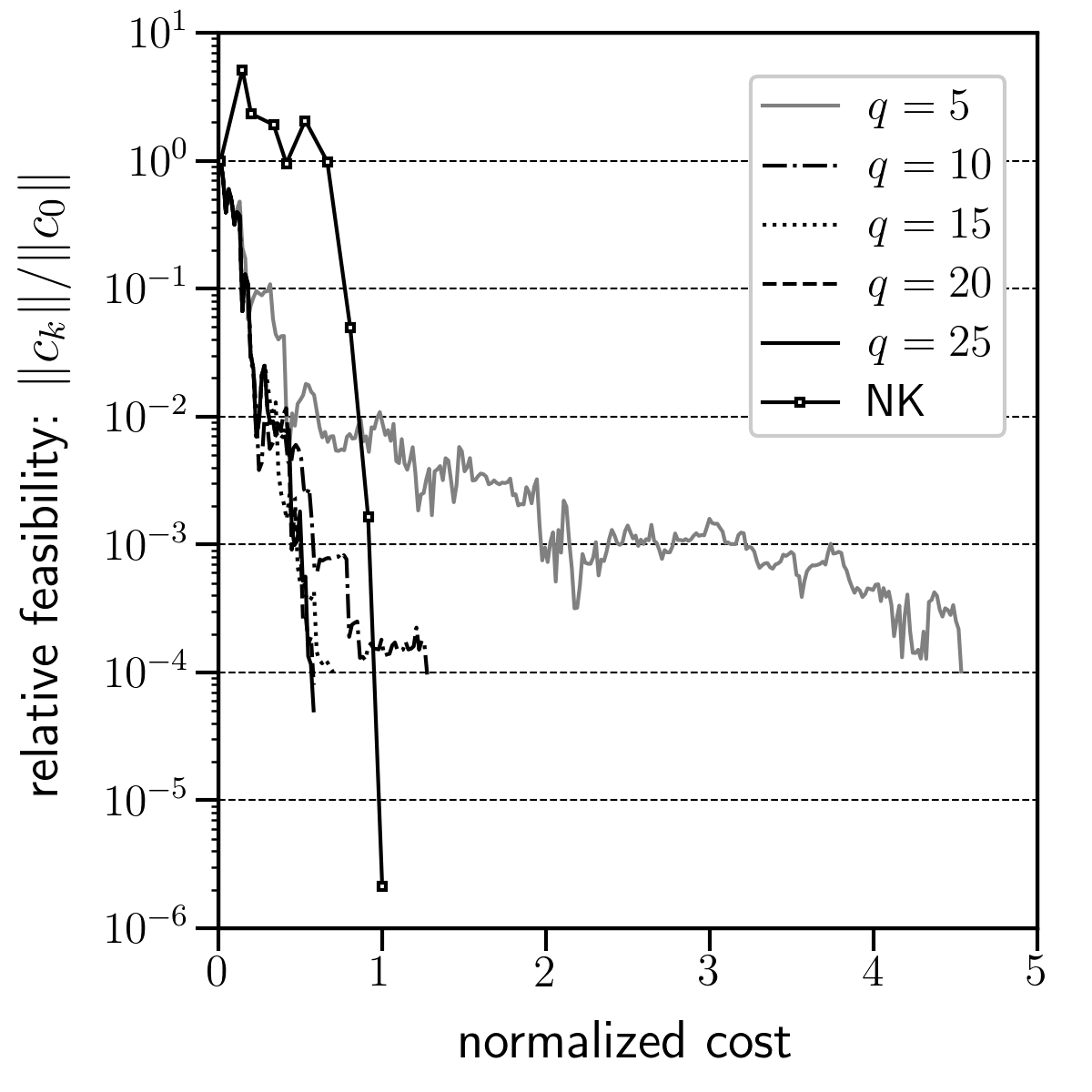}}
  \caption{Convergence histories of MAD, for a range of $q$ values, applied to \eqref{eq:idf} with accurate data.  All variants use $\alpha=0.1$ and $\beta=0.5$.  Results from a Newton-Krylov algorithm are included for reference, and the cost is normalized using the cost of the Newton-Krylov optimization. \label{fig:IDF_accurate}}
\end{figure}

\subsubsection{Performance Using Inaccurate Data}

Next, we repeat the experiments described above, but we loosen the tolerances on the state and adjoint residuals to $10^{-3}$ and $10^{-2}$, respectively.  The convergence histories in this case are plotted in Figure~\ref{fig:IDF_inaccurate}.

With the loose tolerances on the state and adjoint residuals, the objective, constraints, and gradients have sufficiently large errors that conventional globalization methods have difficulties.  This is reflected in the convergence history of the NK algorithm, which shows that the filter-based globalization stalls.  In contrast, the MAD algorithm proceeds without difficulty in the presence of the inaccurate data.

The cost in Figure~\ref{fig:IDF_inaccurate} is normalized by the computational cost of the NK algorithm using \emph{accurate data}.  Thus, it is possible to compare the accurate-data results in Figure~\ref{fig:IDF_accurate} with the inaccurate-data results in Figure~\ref{fig:IDF_inaccurate}.  Comparing the figures, we see that the MAD algorithm is significantly faster when using inaccurate data.  This is because the MAD runs with and without accurate data use approximately the same number of iterations (for this problem); consequently, the optimization cost is directly proportional to the cost of the state and adjoint solves.  Since the state and adjoint solves with loose tolerances are more than twice as fast as those with tight tolerances, this reduction in state/adjoint cost is translated directly into a reduction in MAD optimization cost.

\begin{figure}[tbp]
  \subfigure[relative optimality \label{fig:opt_IDF_inaccurate}]{%
    \includegraphics[width=0.48\textwidth]{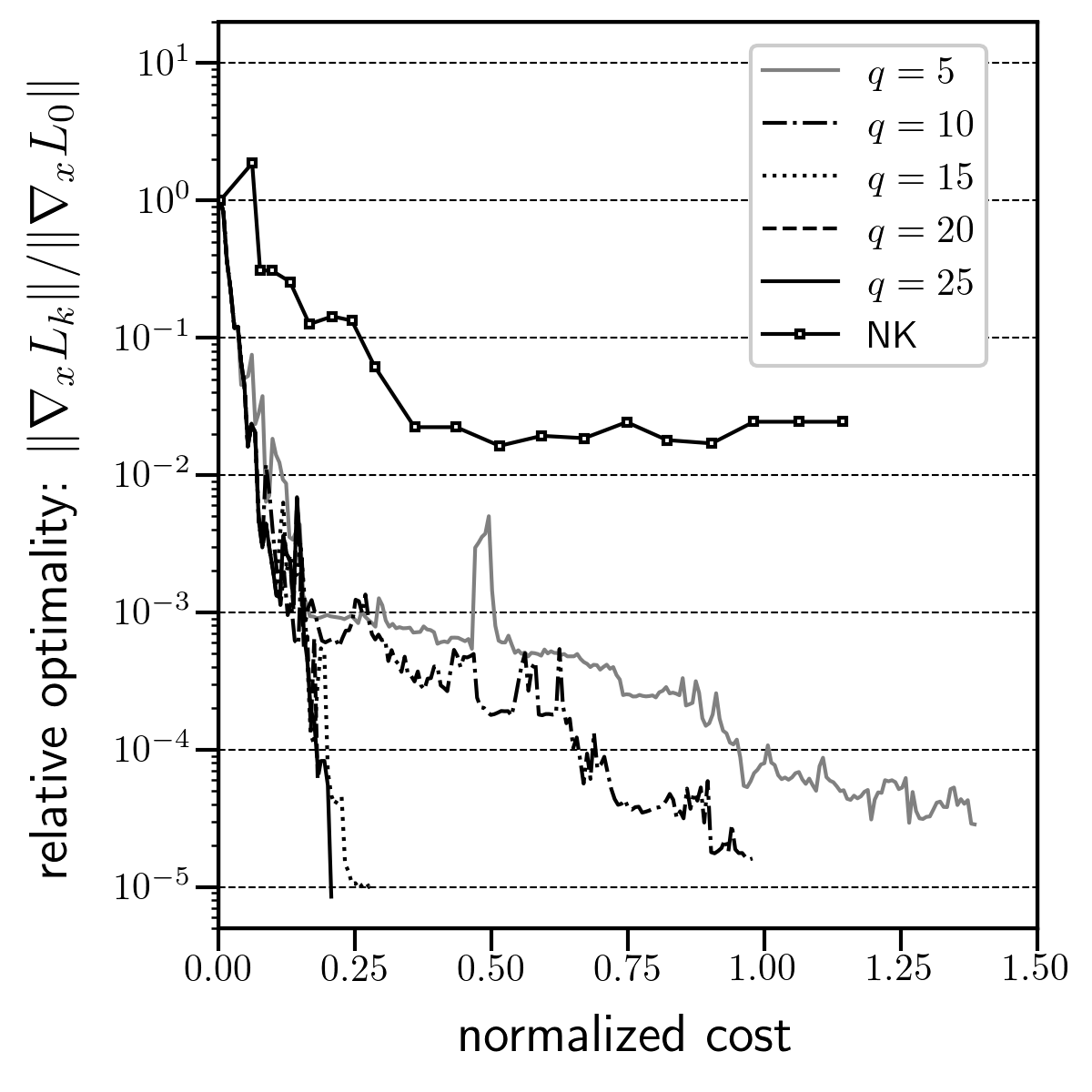}}
  \subfigure[relative feasibility \label{fig:feas_IDF_inaccurate}]{%
    \includegraphics[width=0.48\textwidth]{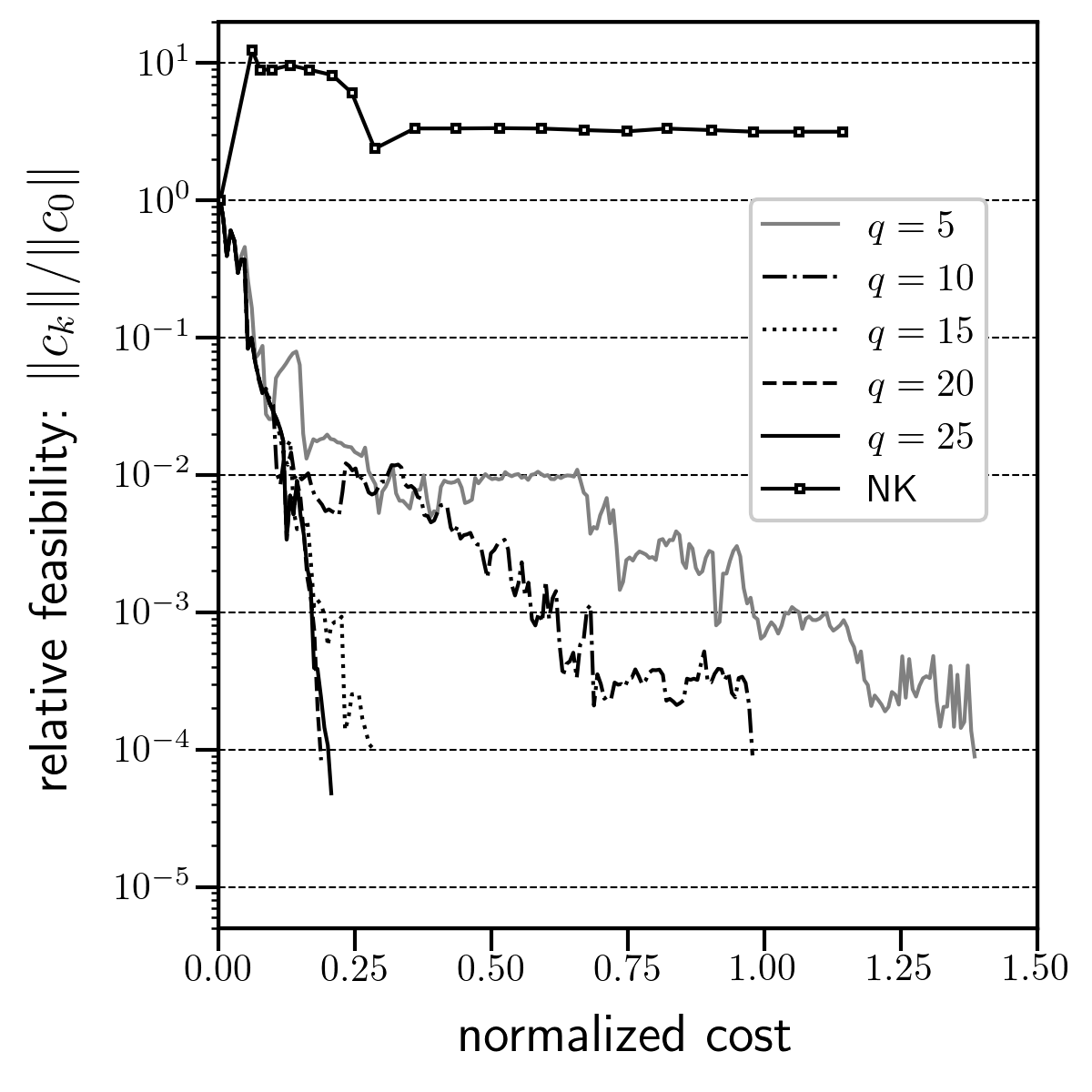}}
  \caption{Convergence histories of MAD, for a range of $q$ values, applied to \eqref{eq:idf} with inaccurate data.  All variants use $\alpha=0.1$ and $\beta=0.5$.  Results from a Newton-Krylov algorithm are included for reference, and the cost is normalized using the cost of the Newton-Krylov optimization \emph{with accurate data}. \label{fig:IDF_inaccurate}}
\end{figure}

\ignore{
\subsection{Optimization Under Uncertainty Problem}

Use the spar-mass optimization problem with stress constraints (inequality constraints!).  Use stochastic collocation for the accurate data, and use MC for inaccurate evaluation.

KS or Lorenz equation (unconstrained?)
- show effectiveness in the presence of noisy derivatives
}

\section{Summary and Future Work}\label{sec:conclude}

Many design optimization problems involve computationally expensive simulations, and while the objective and constraint derivatives may be available, they may be inaccurate or inconsistent. These errors in the data can cause conventional optimization algorithms to fail.  To enable the use of inaccurate/noisy data in optimization, we have proposed a error-tolerant optimization algorithm based on a multisecant quasi-Newton framework.  

The algorithm solves a set of semi-smooth nonlinear equations that are equivalent to the KKT first-order necessary conditions.  By recasting the KKT conditions as nonlinear equations, it is possible to apply multisecant methods directly to the entire set of equations, in contrast to conventional quasi-Newton optimization algorithms that apply the approximation to the Hessian of the Lagrangian only.  The algorithm was enhanced by incorporating preconditioning into the multisecant update equation, and a regularization was introduced to handle nonconvex problems. 

To demonstrate the method, it was applied to a multidisciplinary design optimization problem that has state-dependent constraints and over a hundred variables.  Using optimal parameter values, the multisecant method was found to be competitive with an inexact-Newton-Krylov algorithm when accurate data was provided.  When inaccurate data was used, the performance of the multisecant method was virtually unchanged, whereas the Newton-Krylov algorithm stalled.

The numerical experiments suggest that the proposed multisecant algorithm is a promising method for nonlinearly constrained optimization problems, both with and without errors in the data.  Nevertheless, some important issues for future research remain.  The parameters $\alpha$ and $\beta$ were found to have a significant impact on the success and performance of the algorithm, yet we do not have an automated means of selecting optimal values for these parameters.  Furthermore, the influence of $\beta$ on regularization is counter-intuitive and demands further investigation.  Finally, the algorithm benefits from preconditioning, but there are few effective, general-purpose preconditioners for constrained optimization problems.  

\section{Acknowledgements} 

J. Hicken was funded by the Air Force Office of Scientific Research Award FA9550-15-1-0242 under Dr. Jean-Luc Cambier. The authors gratefully acknowledge this support.

\bibliographystyle{spmpsci} 
\bibliography{./references}

\end{document}